\documentclass[11pt,a4paper]{amsart}

\usepackage{lscape}%para girar el texto de la p?gina
\usepackage{hyperref} %%%% Incluye los links en el dvi
\usepackage{multirow,arydshln}
\usepackage{amsmath,amsthm, amscd, amssymb, amsfonts}
\usepackage{epsfig}
\usepackage{amsmath,amsthm,amssymb, amscd,enumerate}
\usepackage{color}
\usepackage{tikz}
\usetikzlibrary{matrix,decorations.pathreplacing,arrows,chains,positioning,scopes}
\newtheorem*{theoremSN} {Theorem 1}

\numberwithin{equation}{section}
\theoremstyle{plain}
\newtheorem{theorem} {Theorem} [section]
\newtheorem{lemma} [theorem] {Lemma}
\newtheorem{corollary} [theorem] {Corollary}
\newtheorem{proposition} [theorem] {Proposition}

\theoremstyle{definition}

\newtheorem{remark} [theorem] {Remark}

\newtheorem*{exampleSN} {Example}

\newtheorem*{acknow} {Acknowledgments}

\renewcommand \parallel {/\kern-3pt/}

\numberwithin{equation}{section}
\theoremstyle{plain}
\theoremstyle{definition}

\renewcommand \parallel {/\kern-3pt/}

\DeclareMathOperator{\rad}{rad}
\DeclareMathOperator{\aut}{Aut}
\DeclareMathOperator{\der}{Der}
\DeclareMathOperator{\Hom}{Hom}
\DeclareMathOperator{\End}{End}

\begin{document}

\title [The Automorphisms group]{The Automorphisms group of a Current Lie algebra}

\author[Ochoa]{Jes\'us Alonso Ochoa Arango$^\dagger$}
\thanks{$^\dagger$Supported by Project ID: 8433}
\address{\noindent $^\dagger$Departamento de Matem\'aticas, Facultad de ciencias,
Pontificia Universidad Javeriana. Bogot\'a, Colombia.}
\email{jesus.ochoa@javeriana.edu.co}

\author[Rojas]{Nadina Rojas$^\ddagger$}
\thanks{$^\ddagger$Partially supported by \textsc{Conicet}, \textsc{Secyt and FaCEFyN Univ.\thinspace Nac.\thinspace
C\'{o}rdoba}.}
\address{\noindent $^\ddagger$CIEM-CONICET and FCEFyN Universidad Nacional de C\'ordoba. (5000) Ciudad
Universitaria, C\'ordoba-Argentina.}
\email{nadina.rojas@unc.edu.ar}

%    General info
%\date{\today}
         %<-------------------
%%%%%%%%%%%%%%%%%%%%%%%%%%%%%%%5

\maketitle

\begin{abstract}
Let
$\mathfrak{g}$ be a finite dimensional complex Lie algebra and let
$A$ be a finite dimensional complex,  associative and commutative  algebra with unit.
We describe the structure of the derivation Lie algebra
of the current Lie algebra
$\mathfrak{g}_A= \mathfrak{g} \otimes A$, denoted by
$\der(\mathfrak{g}_A)$. Furthermore, we obtain the Levi decomposition of
$\der(\mathfrak{g}_A)$.

As a consequence of the last result, if
$\mathfrak{h}_m$ is the Heisenberg Lie algebra of dimension
$2 m + 1$, we obtain a faithful representation of
$\der(\mathfrak{h}_{m,k})$ of the current
truncated Heisenberg Lie algebra
$\mathfrak{h}_{m,k}= \mathfrak{h}_m \otimes \mathbb{C}[t]/ (t^{k + 1})$ for all
positive integer
$k$.   %<-------------------
\end{abstract}

\section{Introduction}

%**********************************************************************

The automorphism group of a Lie algebra $\mathfrak{g}$, denoted by $\aut(\mathfrak{g})$,
has been  extensively studied due to its important role in several branches of mathematics and physics
in describing \textit{symmetries} associated with \textit{systems} (see for instance \cite{BM, J2, St}).
In \cite{CGS}, the automorphism groups of nilpotent Lie algebras are used to obtain
the classification of such algebras in low dimensions over arbitrary fields. 
It is well known that the Lie algebra of the Lie group $\aut(\mathfrak{g})$ is the Lie algebra $\der(\mathfrak{g})$ (see for instance \cite[Prop. 1.120]{K}).

Let
$A$ be an complex, associative and commutative algebra with unit. The main purpose of this work is to study the derivation Lie algebra of a \emph{current Lie algebra}
$\mathfrak{g}_A := \mathfrak{g} \otimes_{\mathbb{C}} A$ with the bracket
$$
[X_1 \otimes a_1, X_2 \otimes a_2]= [X_1, X_2] \otimes a_1 a_2,
$$
for $X_i \in \mathfrak{g}$, $a_i \in A$ for $i=1, 2$.  In particular, if
$A_k = \mathbb{C}[t]/(t^{k + 1})$ with $k$ a positive integer,  the current Lie algebra
$\mathfrak{g}_k= \mathfrak{g} \otimes A_k$ is known as the \emph{truncated current Lie algebra}. This class of Lie algebra  has been studied in \cite{RT, T}
for semisimple Lie algebras; the authors have built a system of algebraically independent generators of the algebra of invariant polynomial functions on
$\mathfrak{g}_k$ and an affine subspace of ${\mathfrak{g}_k}^*$, transverse to the regular orbits.

Some homological aspects of the truncated current Lie algebras have also been considered in \cite{FGT, H2, HW, Ku}, while some applications to the theory of PDE's
has been given in \cite{CO, MM}. These Lie algebras have also appeared in  the literature associated with Hanlon's conjecture \cite{H1,H2,H3,H4} which assures that, for certain classes of complex Lie algebras, the homology of $\mathfrak{g}_k$ is related to the homology of $\mathfrak{g}$ by mean of an isomorphism
\begin{equation}\label{M_property}
H_{*}(\mathfrak{g}_k) \cong H_{*}(\mathfrak{g})^{\otimes (k + 1)},
\end{equation}
of graded vector spaces.
For semisimple Lie algebras it was shown to be true in \cite{FGT}, while for the Lie algebra of strictly upper triangular matrices was shown to be false in \cite{Ku}. However, despite of the efforts made, the conjecture remains open for the other Lie algebras.

It is well known that the Heisenberg Lie algebra plays an important role in
mathematics and physics.
In \cite{HW} the authors investigate the isomorphism
\eqref{M_property} for the Heisenberg Lie algebra
$\mathfrak{h}_1$ of dimension $3$. They provide evidence that the conjecture will be true for this algebra by proving an special case involving finer gradings on the homology of this Lie algebra, but the conjecture still remains open for this and all other Heisenberg Lie algebras $\mathfrak{h}_m$ of dimension $2m+1$.

We are specially interested in the Hanlon's conjecture of the Heisenberg Lie algebra
$\mathfrak{h}_m$. Based on ideas that we have learned from \cite{CT} and
from the first author of the aforementioned paper, we think that if the
conjecture is true for
$\mathfrak{h}_{m,k}$ then there is a structure of
$\mathfrak{sp}_{m}(\mathbb{C})$-module on the homology
$H_{*}(\mathfrak{h}_{m,k})$ and the structure of this module should be \textit{very transparent}.
In order to introduce, in a convenient way, the corresponding action, we must study the structure of the Lie algebra $\der(\mathfrak{g}_{k})$.
Our study starts with a result due to Zusmanovich in \cite{Z} concerning to a vector space decomposition of
$\der(\mathfrak{g}_A)$ in direct sum of distinguished vector subspaces.
We describe the behavior of the Lie bracket of
$\der(\mathfrak{g}_A)$ with respect to one of these vector spaces.
Let
$A$ be an algebra and let
$\mathfrak{g}$ be a Lie algebra. Let us denote
$\mathfrak{J}$ by the Jacobson radical of
$A$,
$\der(A)$ the set of derivations of
$A$, $\mathfrak{r}$ the radical of
$\mathfrak{g}$,
$\mathfrak{s}$ a Levy factor of
$\mathfrak{g}$ and by
$\mathfrak{z}(\mathfrak{g})$ the center of the Lie algebra
$\mathfrak{g}$. The main result of this paper is the following:

\begin{theoremSN}\label{Teo:Principal}
Let
$A$ be  a complex finite dimensional, associativite and commutative algebra with unit and let
$S$ be a semi-simple subalgebra of
$A$ given by the Wedderburn-Malcev Theorem. Let
$\mathfrak{g}$ be a complex finite dimensional Lie algebra such that the center
$\mathfrak{z}(\mathfrak{g})$ is contained in
$[\mathfrak{g}, \mathfrak{g}]$.  If
$ \der(A)$ is a solvable Lie algebra then the radical of the Lie algebra
$\der(\mathfrak{g}_A)$,
$\widetilde{\mathfrak{r}}$, is
\begin{align*}
\widetilde{\mathfrak{r}}=& \left(\mathfrak{s} \otimes \mathfrak{J}\right) +
 \left(\mathfrak{r} \otimes A\right) + \left(\Hom_{\mathfrak{g}}(\mathfrak{g},\mathfrak{g}) \otimes \der(A)\right)
+\\
       &+ \Hom(\mathfrak{g}/ [\mathfrak{g},\mathfrak{g}], \mathfrak{z}) \otimes \End(A).
\end{align*}
and
$\mathfrak{s} \otimes S$ correspond to a semisimple Lie subalgebra of
$\der(\mathfrak{g}_A)$.
\end{theoremSN}

This results enable us to describe the Levi decomposition of all the truncated current Heisenberg Lie algebras as is illustrated in the following example.

\vspace{3pt}
\begin{exampleSN}
Let $\mathfrak{h}_{1,1}:= \mathfrak{h}_1 \otimes \mathbb{C}[t]/ (t^2)$ be the
truncated current Heisenberg Lie algebra of dimension $6$, with basis
\begin{eqnarray}
\mathcal{B}_{1,1} &=& \left\{e \otimes 1, e \otimes t, f \otimes 1, f \otimes t, z \otimes 1, z \otimes t\right\}\nonumber
\end{eqnarray}
such that
\begin{eqnarray}
[e \otimes t^i, f \otimes t^j] &=&
\begin{cases}
z \otimes t^{i+j}, &\text{ if } i+j \leq 1; \\
0,  & \text{ otherwise},
\end{cases} \nonumber
\end{eqnarray}
and the other brackets vanishes identically, $[e \otimes t^i, z \otimes t^j]= [f \otimes t^i, z \otimes t^j]= 0$ for $i,j= 0, 1$. Then, for every derivation $\textbf{D} \in \der(\mathfrak{h}_{1,1})$, we have
$$
\setlength{\unitlength}{7mm}
\begin{picture}(15,7)(0,-1)
\linethickness{0.3mm}
\put(0.35,0){\line(0,1){5}}
\put(11.75,0){\line(0,1){5}}
\put(15.15,0){\line(0,1){5}}
\put(0.35,0){\line(1,0){.2}}
\put(0.35,5){\line(1,0){.2}}
\put(15.15,5){\line(-1,0){.2}}
\put(15.15,0){\line(-1,0){.2}}
\linethickness{0.1mm}
\multiput(5.5,1.25)(0,.2){20}{\line(0,1){.1}}
\put(0.35,1.25){\line(1,0){14.85}}
\multiput(0.35,3)(.2,0){57}{\line(1,0){.1}}
\scriptsize{
\put(0.5,4.4){$\small{a_{11} + d_{11}}$}
\put(3.5,4.4){$0$}
\put(0.5,3.4){$\small{a_{21} + d_{21}}$}
\put(2.5,3.4){$\small{a_{11} + d_{11} + d_{22}}$}

\put(9.25,3.4){$\small{b_{11}}$}
\put(9.25,4.5){$0$}
\put(6.5,4.5){$\small{b_{11}}$}
\put(6.5,3.4){$\small{b_{21}}$}

\put(1,2.25){$\small{c_{11}}$}
\put(3.5,2.25){$0$}
\put(1,1.5){$\small{c_{21}}$}
\put(3.5,1.5){$\small{c_{11}}$}

\put(1,0.75){$\small{x_{11}}$}
\put(1,0.25){$\small{x_{21}}$}
\put(3.5,0.75){$\small{x_{12}}$}
\put(3.5,0.25){$\small{x_{22}}$}
\put(6.5,0.75){$\small{x_{13}}$}
\put(6.5,0.25){$\small{x_{23}}$}
\put(9.25,0.75){$\small{x_{14}}$}
\put(9.25,0.25){$\small{x_{24}}$}

\put(11.8,0.75){$\small{2d_{11}}$}
\put(11.8,0.25){$\small{2d_{21}}$}
\put(13.75,0.75){$\small{0}$}
\put(13.15,0.25){$\small{2d_{11}+d_{22}}$}

\put(13.5,2.95){$\small{0}$}

\put(8.25,1.5){$\small{-a_{11} + d_{11}+ d_{22}}$}
\put(9.25,2.25){$0$}
\put(5.75,1.5){$\small{-a_{21} + d_{21}}$}
\put(5.75,2.25){$\small{-a_{11} + d_{11}}$}
\put(-1.55,2.5){$[\textbf{D}]_{\mathcal{B}_{1,1}}=$}
}
\end{picture}
$$
and the Levi factor could be identified with $\mathfrak{sl}_2(\mathbb{C})$, or what is the same $\mathfrak{sp}_{1}(\mathbb{C})$, as we will see in section \S\ref{current}.
\end{exampleSN}

The paper is organized as follows. In section \S\ref{preli} we make a carefully analysis of the Lie bracket of the derivations Lie algebra
$\der(\mathfrak{g}_A)$, through the decomposition given by Zusmanovich in \cite{Z}, arriving to Proposition \ref{vector_space_decomposition}. In section \S\ref{Levi} we will show several results concerning the behaviour of derivations algebra
$\der(A)$ defined over a finite dimensional, associative and commutative algebra
$A$ with unit  reducing it to the study of derivations defined over finite dimensional local algebras. In this section we show the main result of the paper Theorem \eqref{Teo:Principal}.

Finally, in \S\ref{current}  we obtain a faithful representation of the Lie algebra $\der(\mathfrak{h}_{m,k})$, with
$\mathfrak{h}_{m,k}$ the current Heisenberg Lie algebra of dimension
$(2m + 1)(k + 1)$ for all positive integer
$k$. This representation will enable us to identify the Levi decomposition of
$\mathfrak{h}_{m,k}$ and to show that its Levi factor is isomorphic to
$\mathfrak{sp}_m(\mathbb{C})$.

%************************************************************************
%************************************************************************

\section{Lie algebra of derivations of a current Lie algebra}\label{preli}

%************************************************************************

Let 
$\mathfrak{g}$ be a complex Lie algebra and let $A$ be an complex, associative and commutative algebra with unit. The tensor product
$\mathfrak{g}_{A}:= \mathfrak{g} \otimes A$ has a Lie algebra structure if we endow it with the Lie bracket
$$
[X_1 \otimes a_1, X_2 \otimes a_2]= [X_1, X_2] \otimes a_1 a_2,
$$
for every $X_i \in \mathfrak{g}$ and and $a_i \in A$.
This Lie algebra is known as the current Lie algebra associate with $\mathfrak{g} \text{ and } A$. In this section we will give a detailed description of the structure of
$\der(\mathfrak{g}_{A})$.

\begin{remark}
In order to fix some notation, if $V, W$ are $\mathfrak{g}$-modules,  we will write
\begin{enumerate}[(i)]
\item $\Hom_{\mathbb{C}}(V, W)$ the \emph{vector space} of all 
$\mathbb{C}$-linear maps $T: V \rightarrow W$, with standard $\mathfrak{g}$-module structure defined by
$$(x \cdot  T)(v):= T(x \cdot v) - x \cdot T(v),$$
for $x \in \mathfrak{g}$ and $v \in V$. In particular, if $V= W$ then we write $\End_{\mathbb{C}}(V)$ instead of $\Hom_{\mathbb{C}}(V, V)$.
\item $\Hom_{\mathfrak{g}}(V, W)$ will denote the vector space of $\mathfrak{g}$-linear maps, that is, the space of 
$\mathbb{C}$-linear maps $T: V \rightarrow W$ such that $T(x \cdot v)= x \cdot T(v)$ for all $x \in \mathfrak{g}$ and $v \in V$, in particular, $\Hom_{\mathfrak{g}}(V, W)$ is the submodule  of $\mathfrak{g}$-invariants of $\Hom_{\mathbb{C}}(V, W)$.

If we consider $\mathfrak{g}$ acting on itself by means of the adjoint action, the $\mathfrak{g}$-invariants of $\End_\mathbb{C}(\mathfrak{g})$ is the set
$$
\Hom_{\mathfrak{g}}(\mathfrak{g},\mathfrak{g})= \{T \in \End(\mathfrak{g}) : ad_x \circ T= T \circ ad_x \text{ for all } x \in \mathfrak{g} \},
$$
wich is also called the \emph{centroid} of $\mathfrak{g}$ (see for instance \cite{J1}).
\end{enumerate}
\end{remark}

\begin{remark}
Let $A$ be an algebra with unit, it is easy to see that $A$ can be identified, as an algebra, with the vector space
$$
\textbf{A}=\left\{f_a \in \End(A): a \in A \quad \text{and} \quad  f_a(b)= ab \right\},
$$
endowed with composition of maps as its product.
\end{remark}

The following lemma can be found in \cite[Theorem 2.1]{Z} and it gives a convenient decomposition, as vector space, of $\der(\mathfrak{g}_A)$.

\begin{lemma}\label{lema23}
Let
$A$ be  a finite dimensional complex, associative and commutative algebra with unit and let
$\mathfrak{g}$ be a finite dimensional complex Lie algebra. Then $\der(\mathfrak{g}_A)$ \emph{as a vector space} is isomorphic to
\begin{eqnarray}\label{eq:Z}
\der(\mathfrak{g}) \otimes A + \Hom_{\mathfrak{g}}(\mathfrak{g},\mathfrak{g}) \otimes \der(A) +
\Hom(\mathfrak{g}/ [\mathfrak{g},\mathfrak{g}], \mathfrak{z}(\mathfrak{g})) \otimes \operatorname{\End}(A).\nonumber
\end{eqnarray}
\end{lemma}
The following proposition shows the behavior of the Lie bracket in $\der(\mathfrak{g}_A)$ with respect to the decomposition \eqref{eq:Z}.

\begin{proposition}\label{vector_space_decomposition}
Let
$A$ and
$\mathfrak{g}$ be as in Lemma \ref{eq:Z}.
Let $\mathfrak{h}$, $W$ and $\mathfrak{k}$ denote
$\der(\mathfrak{g}) \otimes A$,
$\Hom_{\mathfrak{g}}(\mathfrak{g}, \mathfrak{g}) \otimes \der(A)$ and  $\Hom(\mathfrak{g}/ [\mathfrak{g},\mathfrak{g}], \mathfrak{z}(\mathfrak{g})) \otimes \End(A)$, respectively, then
\begin{enumerate}[(a)]
\item $\mathfrak{h}$ is a Lie subalgebra of $\der(\mathfrak{g}_A)$ and
\item $\mathfrak{k}$ is an ideal in
     $\der(\mathfrak{g}_A)$.
\end{enumerate}
Moreover, the Lie bracket in
$\der(\mathfrak{g}_A)$ is governed by the rules indicated in Table \eqref{Tabla:1}.
\begin{table}[!h]
\begin{centering}
$$
\begin{tabular}{||l|c|c||c||}
\hline
\hline
\vbox to .5cm{}
 &\tiny{$\mathfrak{h}$} & \tiny{$W$} & \tiny{$\mathfrak{k}$}\\
\hline
& & & \\
\tiny{$\mathfrak{h}$} &\tiny{$[D_1, D_2] \otimes a_1 a_2$} & \tiny{$\underbrace{[D, T] \otimes (a \rho)}_{\in \; W} - \underbrace{(T \circ D) \otimes \rho(a)}_{\in \; \mathfrak{h}}$} & \tiny{$\underbrace{[D, T] \otimes (a f) - (T \circ D) \otimes a \cdot f}_{\in \;\mathfrak{k}}$}\\
\hline
     &    & &\\
\tiny{$W$} &  & \tiny{$\underbrace{[T_1, T_2] \otimes (\rho_1 \circ \rho_2)}_{\in \; \mathfrak{k}} + \underbrace{(T_2 \circ T_1) \otimes [\rho_1, \rho_2]}_{\in \; W}$} & \tiny{$\underbrace{(T_1 \circ T_2) \otimes (\rho \circ f) - (T_2 \circ T_1) \otimes (f \circ \rho)}_{\in \; \mathfrak{k}}$}\\
\hline
     &    & &\\
\tiny{$\mathfrak{k}$} &  &  & \tiny{$\underbrace{(T_1 \circ T_2) \otimes (f_1 \circ f_2) - (T_2 \circ T_1) \otimes (f_2 \circ f_1)}_{\in \; \mathfrak{k}}$} \\
\hline
\end{tabular}
$$
\caption{Brackets of \eqref{eq:Z}.}
\label{Tabla:1}
\end{centering}
\end{table}
\end{proposition}

\begin{proof}
We first compute the brackets of the decomposition \eqref{eq:Z}.
\begin{enumerate}[(1)]
\item  Let
$D_1, D_2 \in \der(A)$ and let
$a_1, a_2 \in A$. Therefore
\begin{eqnarray}
\nonumber [D_1 \otimes a_1, D_2 \otimes a_2](x \otimes b)  &=&((D_1 \otimes a_1)(D_2 \otimes a_2)\\
\nonumber                                                  & &-(D_2 \otimes a_2)(D_1 \otimes a_1))(x \otimes b) \\
\nonumber                                                  &=&D_1( D_2 (x)) \otimes a_1 (a_2 b) - D_2( D_1 (x)) \otimes a_2 (a_1 b) \\
\nonumber                                                  &=&(D_1( D_2 (x)) - D_2( D_1 (x)) )\otimes (a_1 a_2) b\\
\label{eq:[W1]}                                           &=& ([D_1, D_2] \otimes a_1 a_2) (x \otimes b).
\end{eqnarray}
for every
$x \in \mathfrak{g}$ and $b \in A$.
It follows that
$\mathfrak{h}$ is a Lie subalgebra of
$\der(\mathfrak{g}_A)$.

\medskip

\item Let
      $D  \otimes a \in \mathfrak{h}$ and $T \otimes \rho \in W$. We claim that
      \begin{eqnarray}\label{eq:[W1W2]}
     \nonumber [D \otimes a, T \otimes \rho] &=& [D, T] \otimes (a \rho) - (T \circ D) \otimes \rho(a) \in \mathfrak{h} \oplus W. 
      \end{eqnarray}
      Indeed,
      \begin{eqnarray}
       \nonumber[D \otimes a, T \otimes \rho](x \otimes b)&=& [D \otimes a, T \otimes \rho](x \otimes b)\\
        \nonumber                                         &=& D( T(x)) \otimes a \rho(b) -  T( D(x)) \otimes  \rho(ab)\\
         \nonumber                                        &=& D( T(x)) \otimes a \rho(b) -  T( D(x)) \otimes  \rho(ab)  \\
          \nonumber                                       & & + T( D(x)) \otimes a \rho(b) - T( D(x)) \otimes a \rho(b)  \\
      \label{eq:TW1W2}                                     &=& \left([D, T] \otimes (a \rho) - (T \circ D) \otimes \rho(a)\right)(x \otimes b).
      \end{eqnarray}
      for every $x \in \mathfrak{g}$ and $b \in A$.

      Since
      $A$ is an commutative algebra, we have
      $a \rho \in \der(A)$ and by straightforward calculation we obtain
      $[D, T] \in \Hom_{\mathfrak{g}}(\mathfrak{g}, \mathfrak{g})$ and
      $T \circ D \in \der(\mathfrak{g})$.
      Then, from \eqref{eq:TW1W2}, it follows that
      \begin{eqnarray}
      [D \otimes a, T \otimes \rho] &=& \underbrace{[D, T] \otimes (a \rho)}_{\in W} - \underbrace{(T \circ D) \otimes \rho(a)}_{\in \mathfrak{h}}.\label{eq:TW1W2F}
      \end{eqnarray}

\item Let
      $D \otimes a \in \mathfrak{h}$ and let
      $T \otimes f \in \mathfrak{k}$.  We claim that
      \begin{eqnarray}
      [D \otimes a, T \otimes f] &=& [D, T] \otimes (a f) - (T \circ D) \otimes a \cdot f \in \mathfrak{k}. \label{eq:TW1W3}
      \end{eqnarray}
      Indeed,
\begin{eqnarray}
[D \otimes a, T \otimes f](x \otimes b)&=& \left((D \otimes a)(T \otimes f)-(T \otimes f)(D \otimes a)\right)(x \otimes b) \nonumber\\
& =& D( T(x)) \otimes a f(b) -  T( D(x)) \otimes  f(ab) \nonumber\\
%& =& D( T(x)) \otimes a f(b) - T( D(x)) \otimes a f(b) - T( D(x)) \otimes f(a)b \nonumber\\
& =& \left([D, T] \otimes (a f) - (T \circ D) \otimes a \cdot f\right)(x \otimes b),\nonumber
\end{eqnarray}
for every
$x \in \mathfrak{g}, b \in A$.

Since
      $D \in \der(\mathfrak{g})$, we obtain
      $D\left([\mathfrak{g}, \mathfrak{g}]\right) \subseteq [\mathfrak{g}, \mathfrak{g}]$ and
      $D(\mathfrak{z}(\mathfrak{g})) \subseteq \mathfrak{z}(\mathfrak{g})$.
      It is follows that
      $[D, T], T \circ D \in \Hom\left(\mathfrak{g}/ [\mathfrak{g},\mathfrak{g}] , \mathfrak{z}(\mathfrak{g})\right)$. Thus
     \begin{eqnarray}
      [D \otimes a, T \otimes f] &\in& \mathfrak{k}.\nonumber
      \end{eqnarray}
 \item Let
      $T_i \otimes \rho_i \in W$ for
      $i= 1, 2$. It is clear that
      \begin{eqnarray}\label{eq:TW2}
      [T_1 \otimes \rho_1, T_2 \otimes \rho_2] &=& \underbrace{[T_1, T_2] \otimes (\circ \rho_2)}_{\mathfrak{k}} + \underbrace{(T_2 \circ T_1) \otimes [\rho_1, \rho_2]}_{W}.
      \end{eqnarray}
%      Indeed,
%      \begin{eqnarray}
%      [T_1 \otimes \rho_1, T_2 \otimes \rho_2](x \otimes b)&=& \left([T_1, T_2] \otimes \rho_1 \circ \rho_2 + T_2 \circ T_1 \otimes [\rho_1, \rho_2]\right) (x \otimes b), \nonumber
%      \end{eqnarray}
%      for every
%$x \in \mathfrak{g}$ and $ b \in A$.

Since $\rho_1, \rho_2 \in \der(A)$ we have $[\rho_1, \rho_2] \in \der(\mathfrak{g})$. It is also easy to see that $T_2 \circ T_1 \in \Hom_{\mathfrak{g}}(\mathfrak{g},\mathfrak{g})$ and that $[T_1, T_2] \in \Hom\left(\mathfrak{g}/ [\mathfrak{g}, \mathfrak{g}], \mathfrak{z}(\mathfrak{g})\right)$.  Therefore
      \begin{eqnarray}
      (T_2 \circ T_1) \otimes [\rho_1, \rho_2] \in W &\text{ and }& [T_1, T_2] \otimes (\rho_1 \circ \rho_2) \in  \mathfrak{k}. \nonumber
      \end{eqnarray}

\item Let $T_i \otimes f_i \in \mathfrak{k}$ for $i= 1, 2$, and $T \otimes \rho \in W$. It is not difficult to prove that $[T_1 \otimes f_1, T_2 \otimes f_2]$ and $[T \otimes \rho, T_1 \otimes f_1]$ are elements of $\mathfrak{k}$.
\end{enumerate}
The proof is now complete.
\end{proof}

%************************************************************************

\section{The Levi Decomposition of the Lie algebra of derivations of a current Lie algebra}\label{Levi}

%************************************************************************

Since we are interested in current Lie algebras that has been built from finite dimensional commutative algebras, then it is important to know their structure with the aim to get some information about their derivation algebra
$\der(A)$.

\begin{proposition}\cite[Prop. 8.6 and Thm. 8.7]{AM}\label{structure-finite-algebras}
Let
$A$ be a finite dimensional, associative and commutative algebra with unit over a field
$\mathbb{K}$. Then
$A$ is isomorphic to a direct sum of local algebras, each of these ones satisfying that its only maximal ideal is nilpotent.
\end{proposition}

\begin{lemma}\label{derivation-sum}
Let
$A$ be a finite dimensional, associative and commutative algebra with unit over a field
$\mathbb{K}$. If
$A$ is a direct sums of algebras
$A_1, \dots, A_r$ then
\begin{eqnarray}
\der(A) &\simeq& \der(A_1) \oplus \dots \oplus \der(A_r). \nonumber
\end{eqnarray}
\end{lemma}

\begin{proof}
The proof of the lemma follows directly from the decomposition of $A$ and the Leibniz rule for derivations.
\end{proof}

\begin{remark}
By Proposition \ref{structure-finite-algebras} and
Lemma \ref{derivation-sum} the study of $\der(A)$ boils down to the study of $\der(A_i)$, where $A_i$ is a local algebra.
\end{remark}

\begin{proposition}\label{derlocal}
Let
$A$ be a finite dimensional, associative and commutative local algebra with unit over a field
$\mathbb{K}$ of characteristic zero. Let
$\mathfrak{m}$ be a maximal ideal of
$A$ then
$D(A) \subseteq \mathfrak{m}$ for all
$D \in \der(A)$.
\end{proposition}

\begin{proof}
Let
$D \in \der(A)$, we first prove that
\begin{align*}
D(\mathfrak{m}) &\subset \mathfrak{m}. \label{eq:max}
\end{align*}
Since
$A$ is a finite dimensional, associative and commutative local algebra with unit and
$\mathfrak{m}$ its maximal ideal, we obtain
$\mathfrak{m}$ is the nil-radical of
$A$. Hence every elements of
$\mathfrak{m}$ is nilpotent and every element of $A$ is unit or nilpotent
\cite[page 89]{AM}.

It follows that for every
$x \in \mathfrak{m}$ there exists
$n \in \mathbb{N}$ such that
\begin{align*}
x^n = 0  \text{ and } x^{n-1} \neq 0.
\end{align*}
Therefore
\begin{eqnarray}
0 &=& D(x^n) \nonumber\\
  &\stackrel{D \in \der(A)}{=}& n x^{n-1}D(x). \nonumber
\end{eqnarray}
Thus
$D(x)$ is not unit, hence
$D(x) \in \mathfrak{m}$. In this way,
$D$ induces a derivation
$\overline{D}$ in
$A/\mathfrak{m}$ over
$\mathbb{K}$.

On the other hand, since
$A$ is finite dimensional over a field
$\mathbb{K}$ and
$\mathfrak{m}$ is maximal ideal of
$A$, we obtain
$A/\mathfrak{m}$ is a finite field extension of
$\mathbb{K}$. Hence
$A/\mathfrak{m}$ is a separable extension of
$\mathbb{K}$ \cite[\S 7]{B2} %bourbaki 2
and in consequence
$\der_{\mathbb{K}}(A/\mathfrak{m})= 0$ \cite[Thm. 25.3]{Ma}.

We conclude that
$\overline{D}= 0$, it follows that
$D(A) \subset \mathfrak{m}$.
\end{proof}

\begin{corollary}\label{remark Importante}
Let
$A$ be a finite dimensional, associative and commutative algebra with unit over a field
$\mathbb{K}$. If $\mathfrak{J}$ is the Jacobson radical of $A$ then $D(A) \subseteq \mathfrak{J}$ for all
$D \in \der(A)$.
\end{corollary}
\begin{proof}
It follows from Proposition \ref{structure-finite-algebras}, Lemma \ref{derivation-sum} and
Proposition \ref{derlocal}.
\end{proof}

The \emph{Wedderburn-Malcev Theorem} \cite[Thm. 72.19]{CR}, said that if $A$ is a finite dimensional algebra over a field
$\mathbb{K}$ with \emph{Jacobson radical}
$\mathfrak{J}$ and such that the residue class algebra
$A/\mathfrak{J}$ is separable, then there exists a semisimple subalgebra
$S$ of $A$, such that
$A = S \oplus \mathfrak{J}$ as a vector space.

On the other hand,
the \emph{Levi-Malcev theorem} \cite[Thm. B.2]{K} said that if
$\mathfrak{g}$ is a finite dimensional Lie algebra, then
$\mathfrak{g}$ can be decomposed as
$\mathfrak{s} \otimes_{\pi}\mathfrak{r}$, where
$\mathfrak{s}$ is a semisimple Lie algebra and
$\mathfrak{r} = \rad \mathfrak{g}$ denote the radical of the Lie algebra
$\mathfrak{g}$.
Applying this theorem to
$\der(\mathfrak{g})$ we obtain
$\der(\mathfrak{g}) =  \mathfrak{s} \ltimes \mathfrak{r}$, a semidirect product of its solvable radical
$\mathfrak{r}$ by a semisimple subalgebra
$\mathfrak{s}$ (the \emph{Levi complement} of
$\mathfrak{r}$).

We are ready to state the main results of the paper.
We can now formulate our main results of this section.

\begin{theorem}\label{Teo:Principal}
Let
$A$ be  a complex finite dimensional, associativite and commutative algebra with unit and let
$S$ be a semisimple subalgebra of
$A$ given by the Wedderburn-Malcev Theorem. Let
$\mathfrak{g}$ be a complex finite dimensional Lie algebra such that the center
$\mathfrak{z}(\mathfrak{g})$ is contained in
$[\mathfrak{g}, \mathfrak{g}]$.  If
$ \der(A)$ is a solvable Lie algebra then
the radical of the Lie algebra
$\der(\mathfrak{g}_A)$,
$\widetilde{\mathfrak{r}}$, is
\begin{eqnarray}\label{rad}
\widetilde{\mathfrak{r}}&=& \left(\mathfrak{s} \otimes \mathfrak{J}\right) +
 \left(\mathfrak{r} \otimes A\right) + \left(\Hom_{\mathfrak{g}}(\mathfrak{g},\mathfrak{g}) \otimes \der(A)\right)
+\\
 \nonumber      &&+ \Hom(\mathfrak{g}/ [\mathfrak{g},\mathfrak{g}], \mathfrak{z}(\mathfrak{g})) \otimes \End(A).
\end{eqnarray}
and
$\mathfrak{s} \otimes S$ correspond to a semisimple Lie subalgebra of
$\der(\mathfrak{g}_A)$.
\end{theorem}

\begin{proof}
Combining the Levi-Malcev Theorem and Wedderburn-Malcev Theorem we obtain
\begin{eqnarray}\label{eq:dergA}
\der(\mathfrak{g}) \otimes A &=& \left(\mathfrak{s} \otimes S\right) \oplus
 \left(\mathfrak{s} \otimes \mathfrak{J}\right) \oplus
 \left(\mathfrak{r} \otimes A\right).
\end{eqnarray}
From Lema \ref{lema23},
$$
\der(\mathfrak{g}_A)= \left(\mathfrak{s} \otimes S\right) \oplus \widetilde{\mathfrak{r}}.
$$

Let
$\mathfrak{h}$, $W$ and $\mathfrak{k}$ denote
$\der(\mathfrak{g}) \otimes A$,
$\Hom_{\mathfrak{g}}(\mathfrak{g}, \mathfrak{g}) \otimes \der(A)$ and
$\Hom(\mathfrak{g}/ [\mathfrak{g},\mathfrak{g}], \mathfrak{z}(\mathfrak{g})) \otimes \End(A)$, respectively.

We first prove that
$\widetilde{\mathfrak{r}}$ is an ideal of
$\der(\mathfrak{g}_A)$.
Let
$D \in \der(A)$, from Corollary \ref{remark Importante} and Table \ref{Tabla:1}, we obtain
\begin{eqnarray}\label{ideal-1}
[\mathfrak{h}, W] \subset (\der(\mathfrak{g}) \otimes \mathfrak{J}) +  W.
\end{eqnarray}

Since
$[\mathfrak{s} \otimes \mathfrak{J}, \mathfrak{s} \otimes A] \subseteq \mathfrak{s} \otimes \mathfrak{J}$ and
$[\mathfrak{s} \otimes \mathfrak{J}, \mathfrak{r} \otimes A] \subseteq \mathfrak{r} \otimes \mathfrak{J}$, we obtain
$$
[\mathfrak{h}, \mathfrak{s} \otimes \mathfrak{J} ]\subseteq \mathfrak{s} \otimes \mathfrak{J} + \mathfrak{r} \otimes A.
$$
Combining this with \eqref{ideal-1} and with the results in Table \ref{Tabla:1}, we have
$\widetilde{\mathfrak{r}}$ is an ideal of
$\der(\mathfrak{g}_A)$.

We shall now prove that
$\widetilde{\mathfrak{r}}$ is solvable ideal.
We will denote by the derived series of $I$ with
\begin{eqnarray}
I^{(n)} = [I^{(n-1)}, I^{(n-1)}] \;\; \text{ and } \;\; I^{(0)}= I \nonumber
\end{eqnarray}
for $n \in N$.
The following facts will lead us to the proof of the solvability of
$\widetilde{\mathfrak{r}}$.
\begin{enumerate}[(1)]
\item Since
      $\mathfrak{J}$ is the Jacobson radical of
      $A$, there exists
      $m \in \mathbb{N}$ such that
      $\mathfrak{J}^{m} = 0$, that is,
      \begin{equation}\label{eq:jacobson}
      a_1 \dots a_m = 0 \text{ for all }  a_i \in \mathfrak{J} \text{ with } i=1, \ldots , m
      \end{equation}
      It follows that
      $
      (\der(\mathfrak{g})\otimes \mathfrak{J})^{m} = 0.
      $

\medskip

\item If $T_i \otimes \rho_i \in W$ for $i=1, 2$, from Table \ref{Tabla:1}, we have
      \begin{eqnarray}
      [T_1 \otimes \rho_1, T_2 \otimes \rho_2]& =& \underbrace{[T_1, T_2] \otimes (\rho_1 \circ \rho_2)}_{\in \mathfrak{k}} + \underbrace{(T_2 \circ T_1) \otimes [\rho_1, \rho_2]}_{\in \Hom_{\mathfrak{g}}(\mathfrak{g},\mathfrak{g}) \otimes \der(A)^{(1)}}\;.\nonumber
      \end{eqnarray}
      By hypothesis
      $\mathfrak{z}(\mathfrak{g}) \subseteq [\mathfrak{g}, \mathfrak{g}]$ and from Table \ref{Tabla:1} we can deduce that
      $\mathfrak{k}$ is an abelian Lie subalgebra of
      $\der(\mathfrak{g}_A)$. Hence
      $\mathfrak{k}^{(1)}=0$, thus
      \begin{eqnarray}\label{eq:W1}
      W^{(n)} &\subseteq& \mathfrak{k} + \Hom_{\mathfrak{g}}(\mathfrak{g}, \mathfrak{g}) \otimes \der(A)^{(n)}\;. %\nonumber
      \end{eqnarray}
The hypothesis indicates us that $\der(A)$ is a solvable Lie algebra, then there exists
$k \in \mathbb{N}$ such that
$\der(A)^{(k)}= 0$. Therefore
      \begin{eqnarray}\label{eq:W2}
      W^{(k + 1)} &=& 0.
      \end{eqnarray}

\medskip

\item Since
      $\der(A)$ is a solvable Lie algebra and from \eqref{eq:jacobson} we obtain
\begin{eqnarray}
[\mathfrak{h}, W]^{(l)} &\subseteq& \der(\mathfrak{g}) \otimes \mathfrak{J} + \mathfrak{k}. \nonumber
\end{eqnarray}
Thus, from Table \ref{Tabla:1}
\begin{eqnarray} \label{eq:W1W2}
[\mathfrak{h}, W]^{(k + l)} = 0.
\end{eqnarray}
\end{enumerate}

We conclude from Table \ref{Tabla:1},\eqref{eq:W1},
\eqref{eq:W2} and
\eqref{eq:W1W2} that
$\widetilde{\mathfrak{r}}$ is a solvable ideal.
We conclude that
$\widetilde{\mathfrak{r}}$ is a solvable ideal.

Before we check that
$\widetilde{\mathfrak{r}}$ is maximal, we will prove that
$\mathfrak{s} \otimes S$ is semisimple. In fact, since $S$ is a semisimple algebra over $\mathbb{C}$ and  $A$ is a complex finite dimensional, associative and commutative algebra, then from the \emph{Artin-Wedderburn Theorem} we get that there exits $n \in \mathbb{N}$ such that
\begin{eqnarray}
S &\simeq& \underbrace{\mathbb{C} \oplus \dots \oplus \mathbb{C}}_{n \;\text{copies}}.
\end{eqnarray}
Therefore
$\mathfrak{s}\otimes S$ is isomorphic to direct sum of
$k$ copies of $\mathfrak{s}$ and then it is a complex semisimple Lie algebra.
It follows that
$\left(\mathfrak{s}\otimes S\right) \cap \widetilde{\mathfrak{r}} = 0$, therefore
\begin{eqnarray}
\der(\mathfrak{g}_A) &=& \left(\mathfrak{s}\otimes S\right) \oplus \widetilde{\mathfrak{r}}. \nonumber
\end{eqnarray}
with
$\widetilde{\mathfrak{r}}$ a maximal solvable ideal and $\mathfrak{s} \otimes S$ a semisimple subalgebra.
This completes the proof.
\end{proof}

\begin{corollary}\label{Coro:1}
Under the same hypothesis as Theorem \ref{Teo:Principal} the Levi factor of
$\der(\mathfrak{g}_A)$ is
$$
\mathfrak{s} \otimes S.
$$
\end{corollary}

\begin{exampleSN}
\begin{enumerate}[(1)]
\item If
$A= \mathbb{C} \oplus \mathfrak{J}$, from Corollary \ref{Coro:1}, the Levi factor of
$\der(\mathfrak{g}_A)$ is the Levi factor of
$\mathfrak{g}$.

\medskip

\item Let
      $\mathfrak{g}$ be a semisimple Lie algebra then the radical
      $\widetilde{\mathfrak{r}}$ is
      $$
      \widetilde{\mathfrak{r}}= \left(\mathfrak{g} \otimes \mathfrak{J}\right) +
      \left(\mathbb{C} \otimes \der(A)\right)
      $$
      and the Levi factor of
      $\der(\mathfrak{g}_A)$ is
      $\mathfrak{g} \otimes S$.

\medskip

\item It is clear that if $A$ is a finite dimensional algebra over
$\mathbb{K}$ then $\der_{\mathbb{K}}(A)$ is the Lie algebra of the algebraic group
$\aut_{\mathbb{K}}(A)$, the group of algebra automorphisms of
$A$ \cite[\textsection 2.3, exam. 2]{OV}. In \cite[\textsection, rem. 2.7]{GS} the authors provide us with an example of a finite dimensional algebra $A$ with non solvable group of authomorphisms, the example they've built goes as follows. Let $I$ be the ideal in $\mathbb{K}[x,y,z]$ generated by $x^{2}, xy, y^{2}$ and all the monomials in $x,y,z$ of degree $3$. If we take $A= \mathbb{K}[x,y,z]/I$ then the  identity component of $\aut(A)$ is not solvable and, in consequence,
$\der_{\mathbb{K}}(A)$ isn't.
\end{enumerate}
\end{exampleSN}

%************************************************************************
\section{A Faithful Representations of Derivations Lie Algebra of Current Truncated Heisenberg Lie algebras}\label{current}

%************************************************************************

Let
$m \in \mathbb{N}$ and let
$\mathfrak{h}_m$ be the Heisenberg Lie algebra of dimension
$2m + 1$. Our main interest is in the derivations Lie algebra of truncated current Heisenberg Lie algebra
$\mathfrak{h}_{m,k}:= \mathfrak{h}_m \otimes \mathbb{C}[t]/ (t^{k + 1})$. We denote by
$\mathcal{B}=  \{e_1, \dots, e_m, f_1, \dots, f_m, x\}$ a bases of
$\mathfrak{h}_m$ whose only non-zero brackets are
\begin{align*}
[e_i, f_i]&= Z
\end{align*}
and let
$\mathcal{C}= \{1, t, \dots, t^k\}$ be the canonical basis of
$\mathbb{C}[t]/ (t^{k + 1})$.
Then a bases of
$\mathfrak{h}_{m,k}$ is
\begin{eqnarray}\label{eq:base}
\mathcal{B}_{m,k}= \left\{e_{i,0}, \dots, e_{i,k}, f_{i,0}, \dots, f_{i,k}, x_0, \dots, x_k : \text{ for } i= 1, \dots, m\right\}.\nonumber
\end{eqnarray}
such that
\begin{eqnarray}
[e \otimes t^i, f \otimes t^j] &=&
\begin{cases}
z \otimes t^{i+j}, &\text{ if } i+j \leq k; \\
0,  & \text{ otherwise}
\end{cases} \nonumber
\end{eqnarray}
and
$[e \otimes t^i, z \otimes t^j]= [f \otimes t^i, z \otimes t^j]= 0$
for
$i,j= 0, \dots, k$.

In order to apply here all the results developed in the preceding sections
we need the following facts:
\begin{enumerate}[(1)]
\item \label{Derivations_Heisenberg}
      A faithful representations of the derivations Lie algebra of
the Heisenberg Lie algebra of dimension $2m+1$ in the basis
$\mathcal{B}$. Let
$$
\pi_0: \der(\mathfrak{h}_m) \rightarrow \mathfrak{gl}(2m + 1)
$$
is expressed, in the canonical basis of
$M_{2m+1, 2m+1}$,
$$
\pi_0(D)=  \begin{bmatrix}
A + aI_{2m} & 0\\
x^{t} & 2a
\end{bmatrix},
$$
where
$A \in \mathfrak{sp}_{2m}(\mathbb{C})$,
$I_{2m}$ is the identity matrix of size
$2m$ with complex entries,
$a \in \mathbb{C}$,
$x \in \mathbb{C}^{2n}$ and $0$ is the zero vector in
$\mathbb{C}^{2n}$. From here we can deduce that the Levi factor of
$\der(\mathfrak{h}_{2m+1})$, is given by
$\mathfrak{sp}_{2m}(\mathbb{C})$ and its radical is the ideal of matrices of the form
$
\begin{bmatrix}
aI_{2m} & 0 \\
x^{t} & 2a
\end{bmatrix}
$.

\medskip

\item The regular representation of
$\mathbb{C}[t]/ (t^{k + 1})$ is
$$
R: \mathbb{C}[t]/ (t^{k + 1}) \to \End_{\mathbb{C}}\left( \mathbb{C}[t]/ (t^{k + 1}) \right)
$$
is expressed, in terms of the canonical basis
$\mathcal{C}$, by
$$
R\left(\sum_{i=0}^k a_i t^i\right)= \left[\begin{smallmatrix}
                          a_0  &       &        &        &        &  \\
                          a_1  & a_0    &        &        &   0     & \\
                          a_2  & a_1    & a_0   &        &        &  \\
                          \vdots      & \vdots    & \ddots   & \ddots &        &  \\
                          a_{k-1} & a_{k-2} & \dots & a_1 & a_0 & \\
                          a_k      & a_{k-1}  & \dots   & a_2 & a_1 & a_0
                \end{smallmatrix}\right].
$$

\medskip

\item Let us define an auxiliary linear map
      $\overline{R}: \mathbb{C}[t]/ (t^{k + 1}) \to \mathbb{C}[t]/ (t^{k + 1})$ whose representation in the canonical basis
      $\mathcal{C}$ is given by
$$
\overline{R}\left(\sum_{i=0}^k a_i t^i\right)= \left[\begin{smallmatrix}
                          0  &       &        &        &        &  \\
                          0  & a_1    &        &        &   0     & \\
                          0  & a_2    & 2a_1   &        &        &  \\
                          \vdots      & \vdots    & \ddots   & \ddots &        &  \\
                         0  & a_{k-1} & \dots & (k-2)a_2 & (k-1)a_1 & \\
                          0      & a_{k}  & \dots   & (k-2)a_3 & (k-1)a_2 & ka_1
                \end{smallmatrix}\right].
$$
\end{enumerate}

\begin{theorem}\label{Teo:Derivations}
Let
$m, k$ be the positive integer such that
$m> 0$ and let
$\rho \in \der(\mathfrak{h}_{m,k})$. Then, in the base
$\mathcal{B}$ the linear map
$\rho$ has the matrix representation
$$
\setlength{\unitlength}{8mm}
\begin{picture}(10,5)(0,-1)
\linethickness{0.3mm}
\put(0,0){\line(0,1){3.5}}
\put(6.5,0){\line(0,1){3.5}}
\put(9,0){\line(0,1){3.5}}
\put(0,0){\line(1,0){.2}}
\put(0,3.5){\line(1,0){.2}}
\put(9,3.5){\line(-1,0){.2}}
\put(9,0){\line(-1,0){.2}}
\linethickness{0.1mm}
\multiput(3.25,1.1)(0,.2){12}{\line(0,1){.1}}
\put(0,1){\line(1,0){9}}
\multiput(0,2.25)(.2,0){33}{\line(1,0){.1}}
\scriptsize{
\put(0.35,2.85){$\small{A_1 + R_1+ R_2}$}
\put(4,2.85){$\small{A_2}$}

\put(1,1.5){$\small{A_4}$}
\put(6.55,0.5){$\small{2 R(p) + \overline{R}(q)}$}
\put(7.5,2.25){$\small{0}$}

\put(1.2,0.45){$\ast$}
\put(4.6,0.45){$\ast$}

\put(3.5,1.5){$\small{A_3 + R_1+ R_2}$}
%\put(-1.5,1.75){$[\rho]_{\mathscr{B}}=$}
\put(0,0){$\underbrace{\rule{73pt}{0pt}}_{\text{\tiny{$m(k+1)$}}}$}
\put(3.2,0){$\underbrace{\rule{73pt}{0pt}}_{\text{\tiny{$m(k+1)$}}}$}
\put(6.45,0){$\underbrace{\rule{58pt}{0pt}}_{\text{\tiny{$k+1$}}}$}
}
\end{picture}
$$
such that:
\begin{enumerate}[(a)]
\item For every $r \in \{1,2,3,4\}$ the matrix $A_r$ is a block matrix
      $$
      A_r= \left[
           \begin{smallmatrix}
           [A_r]_{11} & \dots  & [A_r]_{1m} \\
            \vdots    &        & \vdots\\
            [A_r]_{m1}    &  \dots     &  [A_r]_{mm}
           \end{smallmatrix}
           \right]
      $$
where each block $[A_r]_{ij} \in M_{k+1,k+1}$ is a lower triangular matrix. More exactly, for every $i,j= 1, \dots, m$ we have
\begin{itemize}
\item The blocks $[A_r]_{ij}$ are lower triangular Toeplitz matrices;
\item $[A_3]_{ij}= -[A_1]_{ji}$;
\item $[A_2]_{ij}= [A_2]_{ji}$ and
\item $[A_4]_{ij}= [A_4]_{ji}$.
\end{itemize}
\item $R_1$ and
      $R_2$ are the block matrix
      $$
      R_1= I_{m} \otimes R(p) = \left[
           \begin{smallmatrix}
           R(p) &  & 0 \\
                & \ddots& \\
            0    &       & R(p)
           \end{smallmatrix}
           \right] \qquad \text{ and } \qquad
           R_2 = I_{m} \otimes \overline{R}(p)= \left[
           \begin{smallmatrix}
           \overline{R}(q) &  & 0 \\
                & \ddots& \\
            0    &       & \overline{R}(q)
           \end{smallmatrix}
           \right]
      $$
      where $\otimes$ stand for the Kronecker tensor product of matrices.
\end{enumerate}
Moreover, the block matrices $A_r$ can be written as $[A_r]_{ij} = [B_r]_{ij} + [D_r]_{ij}$ where $[D_r]_{ij}$ is diagonal matrix, multiple of the identity, such that the matrices of the form
$$
\setlength{\unitlength}{8mm}
\begin{picture}(10,5)(0,-1)
\linethickness{0.3mm}
\put(0,0){\line(0,1){3.5}}
\put(6.5,0){\line(0,1){3.5}}
\put(9,0){\line(0,1){3.5}}
\put(0,0){\line(1,0){.2}}
\put(0,3.5){\line(1,0){.2}}
\put(9,3.5){\line(-1,0){.2}}
\put(9,0){\line(-1,0){.2}}
\linethickness{0.1mm}
\multiput(3.25,1.1)(0,.2){12}{\line(0,1){.1}}
\put(0,1){\line(1,0){9}}
\multiput(0,2.25)(.2,0){33}{\line(1,0){.1}}
\scriptsize{
\put(1,2.85){$\small{D_1}$}
\put(4.4,2.85){$\small{D_2}$}

\put(1,1.5){$\small{D_4}$}
\put(7.5,2.25){$0$}

\put(1.2,0.45){$0$}
\put(4.6,0.45){$0$}
\put(7.5,0.45){$0$}

\put(4.4,1.5){$\small{D_3}$}
%\put(-1.5,1.75){$[\rho]_{\mathscr{B}}=$}
\put(0,0){$\underbrace{\rule{73pt}{0pt}}_{\text{\tiny{$m(k+1)$}}}$}
\put(3.2,0){$\underbrace{\rule{73pt}{0pt}}_{\text{\tiny{$m(k+1)$}}}$}
\put(6.45,0){$\underbrace{\rule{58pt}{0pt}}_{\text{\tiny{$k+1$}}}$}
}
\end{picture}
$$
becomes a representation of the simplectic Lie algebra $\mathfrak{sp}_{2n}(\mathbb{C})$ and constitutes a Levi factor of $\der(\mathfrak{h}_{m,k})$. In fact, the $D$'s block is equals to $D \otimes I_{k+1}$ with $D \in \mathfrak{sp}_{2n}(\mathbb{C})$.
\end{theorem}

\begin{proof}
From \eqref{Derivations_Heisenberg}, the Levi factor
$\mathfrak{s}$ of
$\der(\mathfrak{h}_{2m+1})$ consists of the matrices of the form
$$
\setlength{\unitlength}{7mm}
\begin{picture}(15,5.4)(0,-1)
\linethickness{0.3mm}
\put(3,0){\line(0,1){5}}
\put(9,0){\line(0,1){5}}
\put(3,0){\line(1,0){.2}}
\put(3,5){\line(1,0){.2}}
\put(9,5){\line(-1,0){.2}}
\put(9,0){\line(-1,0){.2}}
\linethickness{0.1mm}
\multiput(5.75,0.85)(0,.2){21}{\line(0,1){.1}}
\put(3,0.85){\line(1,0){6}}
\put(8,0){\line(0,1){5}}
\multiput(3,3)(.2,0){25}{\line(1,0){.1}}
\scriptsize{
\put(4.25,3.75){$X_1$}
\put(6.5,3.75){$X_2$}
\put(4.25,2){$X_3$}
\put(4.25,0.25){$\small{0}$}
\put(6.75,0.25){$\small{0}$}
\put(8.25,0.25){$\small{0}$}
\put(8.25,2.95){$\small{0}$}
\put(6.5,2){$X_4$}
\put(1.5,2.9){$[X]_{\mathcal{B}}=$}
\put(2.90,0){$\underbrace{\rule{54pt}{0pt}}_{\text{\tiny{$m$}}}$}
\put(5.65,0){$\underbrace{\rule{47pt}{0pt}}_{\text{\tiny{$m$}}}$}
\put(8,0){$\underbrace{\rule{19pt}{0pt}}_{\text{\tiny{$1$}}}$}
}
\end{picture}
$$
such that
$X_4= -X_1^t$, $X_2= X_2^t$ and  $X_3= X_3^t$. On the other hand, the matrices in the radical $\mathfrak{r}$ of $\der(\mathfrak{h}_{2m+1})$ are matrices of the form
$$
\setlength{\unitlength}{7mm}
\begin{picture}(15,5.8)(0,-1)
\linethickness{0.3mm}
\put(3,0){\line(0,1){5}}
\put(9,0){\line(0,1){5}}
\put(3,0){\line(1,0){.2}}
\put(3,5){\line(1,0){.2}}
\put(9,5){\line(-1,0){.2}}
\put(9,0){\line(-1,0){.2}}
\linethickness{0.1mm}
\multiput(5.75,0.85)(0,.2){21}{\line(0,1){.1}}
\put(3,0.85){\line(1,0){6}}
\put(8,0){\line(0,1){5}}
\multiput(3,3)(.2,0){25}{\line(1,0){.1}}
\scriptsize{
\put(4.25,3.75){$a I_{m}$}
\put(6.5,3.75){$\small{0}$}
\put(4.25,2){$\small{0}$}
\put(5.75,0.25){$\small{u}$}
%\put(6.75,0.25){$\small{0}$}
\put(8.25,0.25){$\small{2a}$}
\put(8.25,2.95){$\small{0}$}
\put(6.5,2){$a I_{m}$}
\put(1.5,2.9){$[X]_{\mathcal{B}}=$}
\put(2.90,0){$\underbrace{\rule{54pt}{0pt}}_{\text{\tiny{$m$}}}$}
\put(5.65,0){$\underbrace{\rule{47pt}{0pt}}_{\text{\tiny{$m$}}}$}
\put(8,0){$\underbrace{\rule{19pt}{0pt}}_{\text{\tiny{$1$}}}$}
}
\end{picture}
$$
Other maps associated to the decomposition of
$\der(\mathfrak{h}_{m,k})$, according to Theorem \eqref{Teo:Principal},
are the following ones:
\begin{enumerate}[(i)]
\item Let
      $X \in \mathfrak{h}_m$, the
      $ad_X$ is the form
$$
\setlength{\unitlength}{7mm}
\begin{picture}(15,5.75)(0,-1)
\linethickness{0.3mm}
\put(3,0){\line(0,1){5}}
\put(9,0){\line(0,1){5}}
\put(3,0){\line(1,0){.2}}
\put(3,5){\line(1,0){.2}}
\put(9,5){\line(-1,0){.2}}
\put(9,0){\line(-1,0){.2}}
\linethickness{0.1mm}
\multiput(5.75,0.85)(0,.2){21}{\line(0,1){.1}}
\put(3,0.85){\line(1,0){6}}
\put(8,0){\line(0,1){5}}
\multiput(3,3)(.2,0){25}{\line(1,0){.1}}
\scriptsize{
\put(4.25,3.75){$0$}
\put(6.5,3.75){$\small{0}$}
\put(4.25,2){$\small{0}$}
\put(3.75,0.25){$\small{*}$}
\put(6.75,0.25){$\small{*}$}
\put(8.25,0.25){$\small{0}$}
\put(8.25,2.95){$\small{0}$}
\put(6.5,2){$0$}
\put(0.75,2.9){$[ad_X]_{\mathcal{B}}=$}
\put(2.90,0){$\underbrace{\rule{54pt}{0pt}}_{\text{\tiny{$m$}}}$}
\put(5.65,0){$\underbrace{\rule{47pt}{0pt}}_{\text{\tiny{$m$}}}$}
\put(8,0){$\underbrace{\rule{19pt}{0pt}}_{\text{\tiny{$1$}}}$}
}
\end{picture}
$$
\item If
      $T \in \Hom_{\mathfrak{h}_m}(\mathfrak{h}_m,\mathfrak{h}_m)$ it is easy to check that
$$
\setlength{\unitlength}{7mm}
\begin{picture}(15,5.75)(0,-1)
\linethickness{0.3mm}
\put(3,0){\line(0,1){5}}
\put(9,0){\line(0,1){5}}
\put(3,0){\line(1,0){.2}}
\put(3,5){\line(1,0){.2}}
\put(9,5){\line(-1,0){.2}}
\put(9,0){\line(-1,0){.2}}
\linethickness{0.1mm}
\multiput(5.75,0.85)(0,.2){21}{\line(0,1){.1}}
\put(3,0.85){\line(1,0){6}}
\put(8,0){\line(0,1){5}}
\multiput(3,3)(.2,0){25}{\line(1,0){.1}}
\scriptsize{
\put(4.25,3.75){$b I$}
\put(6.5,3.75){$\small{0}$}
\put(4.25,2){$\small{0}$}
\put(4.25,0.25){$\small{*}$}
\put(6.75,0.25){$\small{*}$}
\put(8.25,0.25){$\small{b}$}
\put(8.25,2.95){$\small{0}$}
\put(6.5,2){$b I$}
\put(1.5,2.9){$[T]_{\beta}=$}
\put(2.90,0){$\underbrace{\rule{54pt}{0pt}}_{\text{\tiny{$m$}}}$}
\put(5.65,0){$\underbrace{\rule{47pt}{0pt}}_{\text{\tiny{$m$}}}$}
\put(8,0){$\underbrace{\rule{19pt}{0pt}}_{\text{\tiny{$1$}}}$}
}
\end{picture}
$$
for some
$b \in \mathbb{C}$.
\item If
      $T \in \Hom(\mathfrak{h}_m/ \mathfrak{z}(\mathfrak{g}), \mathfrak{z}(\mathfrak{g}))$, we obtain
$$
\setlength{\unitlength}{7mm}
\begin{picture}(15,6)(0,-1)
\linethickness{0.3mm}
\put(3,0){\line(0,1){5}}
\put(9,0){\line(0,1){5}}
\put(3,0){\line(1,0){.2}}
\put(3,5){\line(1,0){.2}}
\put(9,5){\line(-1,0){.2}}
\put(9,0){\line(-1,0){.2}}
\linethickness{0.1mm}
\multiput(5.75,0.85)(0,.2){21}{\line(0,1){.1}}
\put(3,0.85){\line(1,0){6}}
\put(8,0){\line(0,1){5}}
\multiput(3,3)(.2,0){25}{\line(1,0){.1}}
\scriptsize{
\put(4.25,3.75){$\small{0}$}
\put(6.5,3.75){$\small{0}$}
\put(4.25,2){$\small{0}$}
\put(4.25,0.25){$\small{*}$}
\put(6.75,0.25){$\small{*}$}
\put(8.25,0.25){$\small{0}$}
\put(8.25,2.95){$\small{0}$}
\put(6.5,2){$\small{0}$}
\put(1.5,2.9){$[T]_{\beta}=$}
\put(2.90,0){$\underbrace{\rule{54pt}{0pt}}_{\text{\tiny{$m$}}}$}
\put(5.65,0){$\underbrace{\rule{47pt}{0pt}}_{\text{\tiny{$m$}}}$}
\put(8,0){$\underbrace{\rule{19pt}{0pt}}_{\text{\tiny{$1$}}}$}
}
\end{picture}
$$
\item If
$\widetilde{\rho} \in \der\left(\mathbb{C}[t]/ (t^{k + 1})\right)$,
 by straightforward calculation we obtain,
\begin{eqnarray}
\widetilde{\rho}(1) &=& 0 \quad \text{ and } \nonumber\\
\widetilde{\rho}\left( t^i\right) &=& i t^{i-1} \widetilde{\rho}(t) \quad \text{ for } i= 1, \dots, k. \nonumber
\end{eqnarray}
Hence, if
$\widetilde{\rho}(t)= \sum_{i=1}^k b_i t^i$ for some polinomial without independent term and
the matrix
$[\widetilde{\rho}]_{\mathcal{C}}= \overline{R}\left(\sum_{i=1}^k b_i t^i\right)$, that is
$$
[\widetilde{\rho}]_{\mathcal{C}}= \left[\begin{smallmatrix}
                          0  &       &        &        &        &  \\
                          0  & b_1    &        &        &   0     & \\
                          0  & b_2    & 2b_1   &        &        &  \\
                          \vdots      & \vdots    & \ddots   & \ddots &        &  \\
                          0 & b_{k-1} & \dots & (k-2)b_2 & (k-1)b_1 & \\
                          0      & b_{k}  & \dots   & (k-2)b_3 & (k-1)b_2 & kb_1
                \end{smallmatrix}\right].
$$
\end{enumerate}
From Theorem \ref{Teo:Principal} and using the Kronecker tensor product and the
matrix shapes obtain above leads to the desired conclusion.
\end{proof}

\begin{corollary}
Let
$m, k$ be the positive integer such that
$m> 0$. Then
$$
\dim \der(\mathfrak{h}_{m,k})= m(2m + 1)(k+1) + 2m(k+1)^2 + 2k + 1.
$$
\end{corollary}

\begin{exampleSN}
Let
$\mathfrak{h}_{1,1}:= \mathfrak{h}_1 \otimes \mathbb{C}[t]/ (t^2)$ be the
current truncated Heisenberg Lie algebra of dimension
$6$ with a basis
$$
\mathcal{B}_{1,1}= \left\{e \otimes 1, e \otimes t, f \otimes 1, f \otimes t, z \otimes 1, z \otimes t\right\}\;.
$$
Then
$$
\setlength{\unitlength}{7mm}
\begin{picture}(15,7)(0,-1)
\linethickness{0.3mm}
\put(0,0){\line(0,1){5}}
\put(11.75,0){\line(0,1){5}}
\put(15,0){\line(0,1){5}}
\put(0,0){\line(1,0){.2}}
\put(0,5){\line(1,0){.2}}
\put(15,5){\line(-1,0){.2}}
\put(15,0){\line(-1,0){.2}}
\linethickness{0.1mm}
\multiput(5.5,1.25)(0,.2){20}{\line(0,1){.1}}
\put(0,1.25){\line(1,0){15}}
\multiput(0,3)(.2,0){59}{\line(1,0){.1}}
\scriptsize{
\put(0.25,4.4){$\small{a_{11} + d_{11}}$}
\put(3.25,4.4){$0$}
\put(0.25,3.4){$\small{a_{21} + d_{21}}$}
\put(2.25,3.4){$\small{a_{11} + d_{11} + d_{22}}$}

\put(9.25,3.4){$\small{b_{11}}$}
\put(9.25,4.5){$0$}
\put(5.75,4.5){$\small{b_{11}}$}
\put(5.75,3.4){$\small{b_{21}}$}

\put(0.25,2.25){$\small{c_{11}}$}
\put(3.25,2.25){$0$}
\put(0.25,1.5){$\small{c_{21}}$}
\put(3.25,1.5){$\small{c_{11}}$}

\put(0.25,0.75){$\small{x_{11}}$}
\put(0.25,0.25){$\small{x_{21}}$}
\put(2.25,0.75){$\small{x_{12}}$}
\put(2.25,0.25){$\small{x_{22}}$}
\put(5.75,0.75){$\small{x_{13}}$}
\put(5.75,0.25){$\small{x_{23}}$}
\put(8.25,0.75){$\small{x_{14}}$}
\put(8.25,0.25){$\small{x_{24}}$}

\put(11.8,0.75){$\small{2d_{11}}$}
\put(11.8,0.25){$\small{2d_{21}}$}
\put(13.5,0.75){$\small{0}$}
\put(13.05,0.25){$\small{2d_{11}+d_{22}}$}

\put(13.25,2.95){$\small{0}$}

\put(8.25,1.5){$\small{-a_{11} + d_{11}+ d_{22}}$}
\put(9.25,2.25){$0$}
\put(5.75,1.5){$\small{-a_{21} + d_{21}}$}
\put(5.75,2.25){$\small{-a_{11} + d_{11}}$}
\put(-1.85,3){$[\textbf{D}]_{\mathcal{B}_{1,1}}=$}
}
\end{picture}
$$
for all
$\textbf{D} \in Der(\mathfrak{h}_{1,1})$.
\end{exampleSN}

%***********************************************************************************************

\begin{acknow}
The first author was supported by Pontificia Universidad Javeriana, Bogot\'a, Colombia under the research project with ID : 8433,
he would like tho thanks to his institution for its support.
The second author gratefully acknowledges the many helpful suggestions of Professor Leandro Cagliero of Universidad
Nacional de C\'ordoba, Argentina
during the preparation of the paper.
\end{acknow}

%***********************************************************************************************

%***********************************************************************************************

\end{document}